\documentclass[12pt]{article}
\usepackage{oldlfont}
\usepackage{enumerate}
\usepackage{amsmath}
\usepackage{amssymb}
\usepackage{amsthm}
\usepackage{mathrsfs}
\usepackage{amscd}
\usepackage{exscale}
\usepackage{latexsym}
\usepackage[all]{xy}
\usepackage{hyperref}
\usepackage{fancyhdr}
\usepackage{layout}
\usepackage{color}
\usepackage{graphicx}

\DeclareMathAlphabet{\mathpzc}{OT1}{pzc}{m}{it}

\begin{document}

\baselineskip=17pt

\pagestyle{headings}

\numberwithin{equation}{section}

\makeatletter                                                           

\def\section{\@startsection {section}{1}{\z@}{-5.5ex plus -.5ex         
minus -.2ex}{1ex plus .2ex}{\large \bf}}                                 


\pagestyle{fancy}
\renewcommand{\sectionmark}[1]{\markboth{ #1}{ #1}}
\renewcommand{\subsectionmark}[1]{\markright{ #1}}
\fancyhf{} 
\fancyhead[LE,RO]{\slshape\thepage}
\fancyhead[LO]{\slshape\rightmark}
\fancyhead[RE]{\slshape\leftmark}

\addtolength{\headheight}{0.5pt} 
\renewcommand{\headrulewidth}{0pt} 

\newtheorem{thm}{Theorem}
\newtheorem{mainthm}[thm]{Main Theorem}

\newcommand{\ZZ}{{\mathbb Z}}
\newcommand{\GG}{{\mathbb G}}
\newcommand{\Z}{{\mathbb Z}}
\newcommand{\RR}{{\mathbb R}}
\newcommand{\NN}{{\mathbb N}}
\newcommand{\GF}{{\rm GF}}
\newcommand{\QQ}{{\mathbb Q}}
\newcommand{\CC}{{\mathbb C}}
\newcommand{\FF}{{\mathbb F}}

\newtheorem{lem}[thm]{Lemma}
\newtheorem{cor}[thm]{Corollary}
\newtheorem{pro}[thm]{Proposition}
\newtheorem{proprieta}[thm]{Property}
\newcommand{\pf}{\noindent \textbf{Proof.} \ }
\newcommand{\eop}{$_{\Box}$  \relax}
\newtheorem{num}{equation}{}

\theoremstyle{definition}
\newtheorem{rem}[thm]{Remark}
\newtheorem*{D}{Definition}

\newcommand{\nsplit}{\cdot}
\newcommand{\G}{{\mathfrak g}}
\newcommand{\GL}{{\rm GL}}
\newcommand{\SL}{{\rm SL}}
\newcommand{\SP}{{\rm Sp}}
\newcommand{\LL}{{\rm L}}
\newcommand{\Ker}{{\rm Ker}}
\newcommand{\la}{\langle}
\newcommand{\ra}{\rangle}
\newcommand{\PSp}{{\rm PSp}}
\newcommand{\U}{{\rm U}}
\newcommand{\GU}{{\rm GU}}
\newcommand{\Aut}{{\rm Aut}}
\newcommand{\Alt}{{\rm Alt}}
\newcommand{\Sym}{{\rm Sym}}

\newcommand{\isom}{{\cong}}
\newcommand{\z}{{\zeta}}
\newcommand{\Gal}{{\rm Gal}}

\newcommand{\F}{{\mathbb F}}
\renewcommand{\O}{{\cal O}}
\newcommand{\Q}{{\mathbb Q}}
\newcommand{\R}{{\mathbb R}}
\newcommand{\N}{{\mathbb N}}
\newcommand{\E}{{\mathcal{E}}}

\newcommand{\DIM}{{\smallskip\noindent{\bf Proof.}\quad}}
\newcommand{\CVD}{\begin{flushright}$\square$\end{flushright}
\vskip 0.2cm\goodbreak}


\vskip 0.5cm

\title{On the minimal set for counterexamples to the local-global principle}
\author{Laura Paladino, Gabriele Ranieri, Evelina Viada\footnote{supported by the  Fond National Suisse}}
\date{  }
\maketitle

\vskip 1.5cm

\begin{abstract}
We prove that only for powers of  $2$ and $3$ could occur counterexamples to the local-global divisibility principle for elliptic curves defined over the rationals. For we refine our previous criterion for the validity of the principle.  We also give an example that shows that the assumptions of our criterion are necessary.
\end{abstract}

\section{Introduction}
This is the third of a series of articles on the \emph{Local-Global Divisibility Problem} in elliptic curves. In our previous articles we proved that for an elliptic curve $\mathcal{E}$  defined over a number field $ k $ with no $k$-torsion points of exact order $ p $ the Local-Global Principle holds for divisibility by powers  $p^n$, unless the field $ k $ contains a special real element. As a nice consequence  of the deep theorem of L. Merel \cite{Mer} we obtain that  there exists a constant ${C} ( k )$, depending only on the degree of  $ k $, such that  the local-global principle holds for divisibility by any power $p^n$ of  primes $p > {C} ( k )$.   The famous effective Mazur's theorem  \cite{Maz2}  proves that ${C} ( \QQ )=7$.  We deduce that counterexamples to the local global principle for elliptic curves over the rationals, can only occur for  ${S}_\Q=\{2,3,5,7\}$. Only for  powers of $2$ there are counterexamples over $\Q$. Unfortunately for powers of $3$ all known counterexamples are over non-trivial extensions of $\Q$, see \cite{DZ2} and \cite{Pal2}. The minimality of such a set is then a natural question. In this article we prove that $S_\Q$ is not minimal and it can be  shrunk to $\widetilde{S}_\Q=\{2,3\}$.   However the question  remains open for $3$.
\begin{thm}\label{minimale}
For any $p \geq 5$ and   any elliptic curve  $ \E $ defined over $\Q$, the local-global principle holds for divisibility by $p^n$.
\end{thm}
 We already know that  for $p\ge 11$ the theorem holds. To exclude  $5$ and $7$,  we use two nice results. Merel \cite{Mer2} proved that  $\Q ( \E[7] ) \neq \Q ( \zeta_7 )$ for elliptic curves defined over $\Q$.
  Kenku \cite{Ken}, proved that  the modular curve $Y_0 ( 125 )$ has no rational points. In other words,
 there are no rational cyclic isogeny of degree $125$ between elliptic curves defined over $\Q$. Our main Theorem \ref{teo11} consists in proving that  if  $\Q ( \E[p] ) \neq \Q ( \zeta_p )$ or  if there exists a cyclic $\Q$-isogeny of order $p^3$, then the local-global principle holds for divisibility by $p^n$. We then apply Theorem \ref{teo11} with $p=7$ and $p=5$ to deduce Theorem \ref{minimale}.


Our approach realizes on the cohomological criterion introduced by R. Dvornicich and U. Zannier.
Denote by ${\mathcal{E}}[p^n]$  the $p^n$-torsion subgroup of ${\mathcal{E}}$ and by $K_n = k({\mathcal{E}}[p^n])$  the number field obtained by adding to $ k $ the coordinates of the $p^n$-torsion points of ${\mathcal{E}}$. Let $G_n = \Gal ( K_n / k )$. \\

{\bf Cohomological Criterion} (Dvornicich and   Zannier \cite{DZ} and \cite{DZ3}): {\it If the local cohomology $H^1_{\rm loc}(G_n,{\mathcal{E}}[p^n])$ is trivial, then there are no counterexamples to the local-global divisibility by $p^n$.  On the other hand, if the local cohomology $H^1_{\rm loc}(G_n,{\mathcal{E}}[p^n])$ is not trivial, then there is a counterexample to the local-global divisibility by $p^n$ in a finite extension $F$ of $k$ such that $F \cap K_n = k$}.\\

The local cohomology is the intersection of the kernels of all restrictions maps $H^1(G_n,{\mathcal{E}}[p^n]) \rightarrow H^1(C,{\mathcal{E}}[p^n])$ as $C$ varies over all cyclic subgroups of $G_n$.

Furthermore they proved:
\begin{thm}[Dvornicich and   Zannier \cite{DZ3}]\label{teo:teo11}  Suppose that $\E$ does not admit any $k$-rational isogeny of degree $ p $, then the local-global principle holds for $p^n$.
\end{thm} They show that if there are no such isogeny, then the cohomology is trivial, and consequently so is the local cohomology.

A theorem of Serre \cite{Ser} proves that such an isogeny exists only for $p\le c( k,\E )$,  where $c ( k, \E )$ is a constant depending on $ k $ and $\E$. The effective theorem of Mazur  \cite{Maz} proves  that rational isogenies exist only for $p\in \{2,3,5,7,11,13,17,19,37,43,67,163\}$, on  elliptic curves over $\Q$.
 Thus the local-global principle holds in general for $p^n$ with  $p>c(k,\E)$ and in  elliptic curves over $\QQ$ it suffices $p>163$.

In \cite{P-R-V2}, we give a stronger principle.
\begin{thm}\label{uno}  If there are no $k$-rational points of  exact prime order $p$  and $ k $ does not contain the field  $\Q ( \zeta_p + \overline{\zeta_p} )$, for any  primitive root of unity $\zeta_p$, then the local-global principle holds for $p^n$.\end{thm}  We prove that under such hypothesis the local cohomology is trivial and we apply the cohomological criterion.

By the famous theorems of Merel  and  Mazur, we deduce that there is a constant $C ( k )$ depending only on the degree of $k$ such that, for $p > C ( k )$, the local-global principle holds for divisibility by  $p^n$.
In addition, $C ( \QQ )=7$. Thus only powers of $\{2,3,5,7\}$ might give counterexamples.\\

In the present paper we refine Theorem \ref{uno}. As a consequence we exclude $5$ and $7$.

\begin{thm}[Main Theorem]\label{teo11}
Let $ p $ be a prime number and let $ n $ be a positive integer.
Let $\E$ be an elliptic curve defined over a number field $ k $, which does not contain the field  $\Q ( \zeta_p + \overline{\zeta_p} )$.
Suppose that at least one of the following conditions holds:
\begin{enumerate}
\item  $\E$ does not admit any $k$-rational torsion point of exact order $ p $;
\item $k ( \E[p] ) \neq k ( \zeta_p )$;
\item There does not exist a cyclic $k$-isogeny of degree $p^3$, between two elliptic curves defined over $ k $ and $k$-isogenous to $\E$.
\end{enumerate}
Then, the local-global principle holds for divisibility by $p^n$.
\end{thm}

In the last section we construct an example which shows that the assumption that $k$ does not contain $\Q ( \zeta_p + \overline{\zeta_p} )$ is necessary in our Main Theorem.
To prove the theorem we must understand in details the structure of the local cohomology. In part we profit from our previous results. In addition we describe the cohomology of the diagonal subgroup and we construct special elements in the local cohomology (see sections \ref{sec4} and  \ref{sec5}). The details  of the proof of the main theorem are given in section \ref{strutturadim}.

\section{Previous results and notations}
\label{sec2}
Let $k$ be a number field and let $\mathcal{E}$ be an elliptic curve defined over $k$.
Let $ p $ be a prime.
For every positive integer $ n $, we denote by ${\mathcal{E}}[p^n]$  the $p^n$-torsion subgroup of ${\mathcal{E}}$ and by $K_n = k( {\mathcal{E}}[p^n] )$  the number field obtained by adding to $ k $ the coordinates of the $p^n$-torsion points of ${\mathcal{E}}$.
By Silverman \cite[Chapter III, Corollary 8.1.1]{Sil}, the field $ K_n $ contains a primitive $p^n$th root of unity $\zeta_{p^n} $.
Let $G_n = \Gal ( K_n / k )$.
As usual, we shall view  $\E[p^n]$ as $ \Z/p^n\Z \times \Z/p^n\Z $ and consequently we shall represent $G_n$ as a subgroup of $\textrm{GL}_2 ( \Z / p^n \Z )$, denoted by the same symbol.

We list the results of \cite{P-R-V2} that we shall use in the following sections and we fix our notations.

\begin{lem}[ \cite{P-R-V} Lemma 8]\label{lem:lem14}
Suppose  that $H^1 ( G_n, \E[p^n] ) \neq 0$.
Then either
\begin{equation*}
G_1=\langle \rho \rangle  \,\,\,\,\quad {\rm or} \,\,\,\,\quad
   G_1 =\langle \rho,  \sigma \rangle,
\end{equation*}
where $\rho = {\lambda_1\,\,\, 0 \choose 0 \,\,\, \lambda_2}$ is either the identity or a diagonal matrix with $\lambda_1 \neq \lambda_2 \mod (p)$ and $\sigma={1\,\,\, 1 \choose 0\,\,\, 1}$, in a suitable basis of $\E[p]$.
\end{lem}

By~\cite[Lemma 9]{P-R-V2}, we can choose a basis $\{ Q_1, Q_2 \}$ of $\E[p^n]$ in which a lift $\rho_n$ of $\rho$ to $G_n$ is diagonal and the order of $\rho_n$ is equals to the order of $\rho$. From now on we fix such a basis.
Then
\[
\rho_n =
\left(
\begin{array}{cc}
\lambda_{1, n} & 0 \\
0 & \lambda_{2, n}\\
\end{array}
\right),
\]
with $\lambda_{i, n} \equiv \lambda_i$ modulo $p$ .

This choice of a basis of $\E[p^n]$, determine  the  embedding of $G_n$ in $\GL_2 ( \Z / p^n \Z )$.
We consider the following subgroups of $G_n$:\\
\noindent    the  subgroup ${\mathpzc{sU}}_n$  of strictly upper triangular matrices in $G_n$;\\
    the subgroup ${\mathpzc{sL}}_n$ of strictly lower triangular matrices in $G_n$;\\
    the      subgroup ${\mathpzc{D}}_n$ of   diagonal   matrices in $G_n$.\\

We recall  useful  properties of such subgroups.

\begin{lem}[ \cite{P-R-V2} Proposition 11 and 16]\label{lem:lem11}
Assume that $ H^1 ( G_n, \E[p^n] )\not=0 $ and that the order of $\rho$ is at least $3$. Then:
\begin{enumerate}
\item $G_n$ is generated by matrices of ${\mathpzc{D}}_n$, ${\mathpzc{sU}}_n$ and ${\mathpzc{sL}}_n$;
\item $H^1_{{\rm loc}} ( \langle \rho_n, {\mathpzc{sL}}_n \rangle, \E[p^n] ) = 0$ and $H^1_{{\rm loc}} ( \langle \rho_n, {\mathpzc{sU}}_n \rangle, \E[p^n] ) = 0$.
\end{enumerate}
\end{lem}

Finally, we define
\[
\tau_L =
\left(
\begin{array}{cc}
1 & 0 \\
p^j & 1 \\
\end{array}
\right)
\]
a generator of ${\mathpzc{sL}}_n$, where $j \geq 1$ is an integer, and
\[
\tau_U =
\left(
\begin{array}{cc}
1 & p^i \\
0 & 1 \\
\end{array}
\right)
\]
a generator of ${\mathpzc{sU}}_n$, where $i \in \N$.

\section{Structure of the proof of the Main Theorem }
 \label{strutturadim}
The proof of Theorem \ref{teo11} realises on   the two main Propositions  \ref{pro14} and \ref{pro15} presented  in section \ref{sec3}.
In the following remark we compare the hypotheses of the main theorem and of these propositions.
\begin{rem}
\label{ipotesi} Assume that  $H^1 ( G_n, \E[p^n] ) \neq 0$. Then by Lemma \ref{lem:lem14},    $G_1$ is either generated by $\rho$ or it is generated by $\rho$ and $\sigma$. \\
If   $\E$ admits a $k$-rational point of order $p$, then,
 in a suitable basis of $\E[p]$, we have
\[
\rho =
\left(
\begin{array}{cc}
1 & 0 \\
0 & \lambda_2\\
\end{array}
\right)
\]
with $\lambda_2 \in \Z / p \Z^\ast$.
Indeed,  $G_1 $ fixes the  $k$-rational points of order $p$.  If $G_1 = \langle \rho,\sigma \rangle$, then $\lambda_1=1$. If $G_1=\langle \rho \rangle$, then either $\lambda_1$ or $\lambda_2$ equal to $1$.
Eventually permuting the basis, we can always suppose that $\lambda_1=1$. Thus
\[
\rho_n =
\left(
\begin{array}{cc}
1 & 0 \\
0 & \lambda_{2, n}\\
\end{array}
\right)
\]
 with $\lambda_{2, n} \equiv \lambda_2 \mod ( p )$ and the first eigenvalue equals to $1$ because $\rho$ and $\rho_n$ have the same order.
If $k$ does not contain the field $\Q(\zeta_p+ \overline{\zeta}_p)$, then the order of $\rho$ is $\geq 3$. In fact, by Lemma \ref{lem:lem14}, the order of $\rho$ is the largest integer relatively prime to $p$ that divides $\vert G_1 \vert$.
In addition $[k ( \zeta_p ) : k] \mid \vert G_1 \vert$, $[k ( \zeta_p ) : k] \mid p-1$  and $[k ( \zeta_p ) : k] \geq 3$.
Thus $\rho$ has order $\geq 3$.
\end{rem}
 From Proposition \ref{pro14} below, we deduce the following theorem. By the mentioned work of Merel this excludes powers of $7$ as potential counterexamples over $\Q$.
 \begin{thm}
\label{sette}
Assume that:
\begin{itemize} \item[] $H^1 ( G_n, \E[p^n] ) \neq 0$,
\item[] $ k $
 does not contain the field  $\Q ( \zeta_p + \overline{\zeta_p} )$, \item[]  there is a $k$-rational point of  exact prime order $p$, and \item[] $k ( \E[p] ) \neq k ( \zeta_p ).$ \end{itemize} Then  the local cohomology $H^1_{{\rm loc}} ( G_n, \E[p^n] ) $ is trivial.
\end{thm}
\begin{proof}
In view of Remark \ref{ipotesi}, we can take $\lambda_1=1$. In addition $\rho$ has order at least $3$.

Suppose that $G_1$ is cyclic generated by $\rho$.
Since the first  eigenvalue of $\rho$ is $1$, the only element of $G_1$ with determinant $1$ is the identity.
The basic properties of the Weil pairing (see \cite[Chapter III.8]{Sil}) entail that for every $\tau \in G_1$, we have $\tau ( \zeta_p ) = \zeta_p^{{\rm det} ( \tau )}$.
Thus the unique element of $G_1$ that fixes $\zeta_p$ is the identity.
Since $k ( \zeta_p ) \subseteq K_1$, we get $K_1 = k ( \zeta_p )$, contradicting the hypothesis.
Thus $G_1$ is not cyclic.   By Proposition \ref{pro14}, we obtain $H^1_{{\rm loc}} ( G_n, \E[p^n] ) = 0$.
\end{proof}

\begin{cor}
For any elliptic curve $\E$ over the  rationals, there are no counterexample to the local-global divisibility by powers of $7$.
\end{cor}
\begin{proof} If the global cohomology is trivial then also the local cohomology is trivial and by the Cohomological Criterion there are no counterexamples. We can then assume  $H^1 ( G_7, \E[7^n] ) \neq 0$. Of course $\Q$ does not contain the element $\zeta_7+ \overline{\zeta}_7$. If there are no rational points  of exact order $7$, then we can apply Theorem \ref{uno}. So we can suppose that  there is  a rational point  of exact order $7$.
 Merel \cite{Mer2} proved that  $\Q ( \E[7] ) \neq \Q ( \zeta_7 )$ for elliptic curves defined over $\Q$. Apply Theorem \ref{sette} and the Cohomological Criterion to conclude the proof.
\end{proof}
By Proposition \ref{pro15} below,  we deduce the following theorem. By the mentioned work of Kenku this exclude powers of $5$ as potential counterexamples over $\Q$.
\begin{thm}
\label{cinque}
Assume that: \begin{itemize} \item[] $H^1 ( G_n, \E[p^n] ) \neq 0$, \item[]  $ k $ does not contain the field  $\Q ( \zeta_p + \overline{\zeta_p} )$, \item[] there is a $k$-rational points of  exact prime order $p$,\item[]  $k ( \E[p] ) = k ( \zeta_p )$ and
\item[] there does not exist a cyclic $k$-isogeny of degree $p^3$, between two elliptic curves defined over $ k $ and $k$-isogenous to $\E$.
\end{itemize}  Then  the local cohomology $H^1_{{\rm loc}} ( G_n, \E[p^n] ) $ is trivial. \end{thm}

\begin{proof}
Suppose on the contrary that $H^1_{{\rm loc}} ( G_n, \E[p^n] ) \not= 0$.  Then by Remark \ref{ipotesi} and Proposition \ref{pro15} all the elements of $G_2$ are either lower triangular or upper triangular.
Then $\E[p^2]$ has a cyclic $G_2$-submodule $C_1$ of order $p^2$.
In addition, by  Proposition \ref{pro14}, $G_1$ is cyclic, generated by the diagonal matrix $\rho$.
Thus  $\E[p]$ has a $G_1$-submodule $C_2$ of order $p$ such that $C_1 \cap C_2 = \{ 0 \}$.
By~\cite[Remark III 4.13.2]{Sil} there exist elliptic curves $\E_1, \E_2$ defined over $k$ and cyclic $k$-isogenies $\phi_i \colon \E \rightarrow \E_i$ with $\ker ( \phi_i ) = C_i$
and  $\deg ( \phi_i ) = \vert \ker ( \phi_i ) \vert$. Then  the isogeny
\[
\phi = \phi_1 \circ \widehat{\phi_2} \colon \E_2 \rightarrow \E_1
\]
(where, as usual, $\widehat{\phi_2}$ is the dual isogeny of $\phi_2$) has degree $p^3$.
Let us show that $\phi$ is cyclic.
Since $\ker ( \phi ) \subseteq \E_2[p^3] \cong \Z / p^3 \Z \times \Z / p^3 \Z$, it is sufficient to prove that $\E_2[p] \not \subseteq \ker ( \phi )$.
Recall that
\begin{equation}\label{rel1}
\ker ( \phi ) = \{ P \in \E_2 \ {\rm s.t.} \ \widehat{\phi_2} ( P ) \in \ker ( \phi_1 ) \}.
\end{equation}
As $|\ker ( \widehat{\phi_2} )|=p$, there exists $P \in \E_2[p]$  such that $\widehat{\phi_2} ( P ) \neq 0$. Then
 $\phi_2 \circ \widehat{\phi_2} ( P ) = p P = 0$ and  $\widehat{\phi_2} ( P ) \in \ker ( \phi_2 ) = C_2$.
Since $C_1 \cap C_2 = \{ 0 \}$, $\widehat{\phi_2} ( P ) \not \in \ker ( \phi_1 )$.
By relation~(\ref{rel1}), we get $P \not \in \ker ( \phi )$.
Then $\E_1[p] \not \subseteq \ker ( \phi )$ and $\phi$ is cyclic of order $p^3$ contradicting the hypothesis. So the local cohomology must be trivial.

\end{proof}

\begin{cor}
For any elliptic curve $\E$ over the  rationals, there are no counterexample to the local-global divisibility by power of $5$.
\end{cor}
\begin{proof} If the global cohomology is trivial then also the local cohomology is trivial and by the Cohomological Criterion there are no counterexamples. We can then assume  $H^1 ( G_5, \E[5^n] ) \neq 0$. Of course $\Q$ does not contain the element $\zeta_5+ \overline{\zeta}_5$. If there are no rational points  of exact order $5$, then we can apply Theorem \ref{uno}. So we can suppose that  there is  a rational point  of exact order $5$. If  $\Q ( \E[5] ) \neq \Q ( \zeta_5 )$ we apply  Theorem \ref{sette} and the Cohomological Criterion to conclude the proof. So we can assume that $\Q ( \E[5] ) = \Q ( \zeta_5 )$.   Kenku \cite{Ken}, proved that  the modular curve $Y_0 ( 125 )$ has no rational points. In other words,
 there are no rational cyclic isogeny of degree $125$ between elliptic curves defined over $\Q$. So the hypothesis of Theorem \ref{cinque} are satisfied. Then the local cohomology is trivial and by the Cohomological Criterion there are no counterexamples.
\end{proof}

   {\bf Conclusion of the Proof of Theorem \ref{teo11}.}
   If condition 1. holds, then by Theorem \ref{uno}  there are no counterexamples.
   Assume that condition 1. does not hold and condition 2. holds. Then there is a $k$-rational point of  exact prime order $p$ and $K_1 \not= k ( \zeta_p )$. By Theorem    \ref{sette} the local cohomology is trivial.
   Assume that condition 1. and 2. do not hold and condition 3. holds. Then  there is a $k$-rational points of  exact prime order $p$ and $K_1 = k ( \zeta_p )$. By Theorem    \ref{cinque} the local cohomology is trivial.

   Apply the Cohomological Criterion   to conclude that the  local-global divisibility by $p^n$ holds.

\section{The triviality of the Local Cohomology for Diagonal Matrices.}
\label{sec4}
In this section we prove that for the the diagonal group $H^1_{{\rm loc}} ( {\mathpzc{D}}_n, \E[p^n] ) = 0$. We recall two immediate properties on the local cohomology of the direct sum and normal subgroups. For convenience  we present here a proof. In a final version we would prefer to let it as an exercise.

\begin{lem}\label{lem11}
Let $ \Gamma $ be a group and let $M, N$ be $\Gamma$-modules.
Let $\pi_M$ (respectively $\pi_N$) be the projection from $ M \bigoplus N $ to $M$ (respectively to $N$).
The isomorphism $\phi \colon H^1 ( \Gamma, M \bigoplus N ) \rightarrow H^1 ( \Gamma, M ) \bigoplus H^1 ( \Gamma, N )$ that sends
$[Z] \in H^1 ( \Gamma, M \bigoplus N )$ to $( [\pi_M \circ Z], [\pi_N \circ Z] )$, induces an isomorphism
\[
\phi_{{\rm loc}} \colon H^1_{{\rm loc}} ( \Gamma, M \bigoplus N ) \rightarrow H^1_{{\rm loc}} ( \Gamma, M ) \bigoplus H^1_{{\rm loc}} ( \Gamma, N ).
\]
\end{lem}

\DIM Let $ Z $ be a cocycle from $\Gamma$ to $M \bigoplus N$, which satisfies the local conditions.
Let us denote $Z^M = \pi_M \circ Z$ and $Z^N = \pi_N \circ Z$.
Then for every $\gamma \in \Gamma$, there exists $( m_\gamma, n_\gamma ) \in M \bigoplus N$ such that
\[
Z_\gamma = \gamma ( m_\gamma, n_\gamma ) - ( m_\gamma, n_\gamma ) = ( \gamma m_\gamma - m_\gamma, \gamma n_\gamma - n_\gamma ).
\]
Thus, for every $ \gamma $, we have $Z^M_\gamma = \gamma m_\gamma - m_\gamma$ and $Z^N_\gamma = \gamma n_\gamma - n_\gamma$.
Therefore $Z^M$ and $Z^N$ satisfy the local conditions, and so $ \phi_{{\rm loc}} $ sends $H^1_{{\rm loc}} ( \Gamma, M \bigoplus N )$ to $H^1_{{\rm loc}} ( \Gamma, M ) \bigoplus H^1_{{\rm loc}} ( \Gamma, N )$.
Since $ \phi $ is injective, also $ \phi_{{\rm loc}} $ is injective.
Let $W^M, W^N$ be cocycles such that $[W^M] \in H^1_{{\rm loc}} ( \Gamma, M )$ and $[W^N] \in H^1_{{\rm loc}} ( \Gamma, N )$.
Then, for every $\gamma \in \Gamma$ there exists $a_\gamma \in M, b_\gamma \in N$, such that $W^M_\gamma = \gamma a_\gamma - a_\gamma$, $W^N = \gamma b_\gamma - b_\gamma$.
Thus also the cocycle $W \colon \Gamma \rightarrow M \bigoplus N$ that sends $\gamma$ to $( W^M_\gamma, W^N_\gamma )$ satisfies the local conditions.
In fact, for every $\gamma \in \Gamma$, we have $W_\gamma = \gamma ( a_\gamma, b_\gamma ) - ( a_\gamma, b_\gamma )$.
Then $[W] \in H^1_{{\rm loc}} ( \Gamma, M \bigoplus N )$.
Since, by definition, $\phi_{{\rm loc}} ( [W] ) = ( [W^M], [W^N] )$, then $\phi_{{\rm loc}}$ is surjective.
\CVD

Let $ \Gamma $ be a group, let $M$ be a $\Gamma$-module and let $\Delta$ be the subgroup of $\Gamma$ that fixes $M$.
In the following lemma we prove that $H^1_{{\rm loc}} ( \Gamma/ \Delta, M ) \cong H^1_{{\rm loc}} ( \Gamma, M )$.

\begin{lem}\label{lemnuovo}
Let $ \Gamma $ be a group and let $M$ be a $\Gamma$-module.
Set $\Delta$ the subgroup of $ \Gamma $ formed by the elements that fix $M$.
Then $H^1_{{\rm loc}} ( \Gamma/ \Delta, M ) \cong H^1_{{\rm loc}} ( \Gamma, M )$ and the isomorphism is induced by the inflation.
\end{lem}

\DIM First observe that $\Delta$ is normal in $\Gamma$.
Then we have the inflation-restriction sequence
\[
0 \rightarrow H^1 ( \Gamma / \Delta, M ) \rightarrow H^1 ( \Gamma, M ) \rightarrow H^1 ( \Delta, M ).
\]
Since a cocycle satisfying  the local conditions relative to $\Gamma$, it satisfies them to any subgroup of $\Gamma$, the restriction sends $H^1_{{\rm loc}} ( \Gamma, M )$ to $H^1_{{\rm loc}} ( \Delta, M )$.
Observe that $H^1_{{\rm loc}} ( \Delta, M ) = 0$ because $\Delta$ acts trivially over $M$.
Thus $H^1_{{\rm loc}} ( \Gamma, M )$ is contained in the kernel of the restriction, which is the image of $H^1 ( \Gamma / \Delta, M )$ by the inflation.
Then every $[Z] \in H^1_{{\rm loc}} ( \Gamma, M )$ comes from $[Y] \in H^1 ( \Gamma / \Delta, M )$.
We now show that $[Y] \in H^1_{{\rm loc}} ( \Gamma / \Delta, M )$.
If we choose $Y$ such that $Y_{\overline{\gamma}} = Z_\gamma$, for every $\gamma \in \Gamma$ (where $\overline{\gamma}$ is the class of $ \gamma $ modulo $\Delta$), we have that the inflation of $[Y]$ is $[Z]$ and $Y$ satisfies the local conditions.
Thus $[Y] \in H^1_{{\rm loc}} ( \Gamma / \Delta, M )$ and so the inflation sends $H^1_{{\rm loc}} ( \Gamma / \Delta, M )$ onto $H^1_{{\rm loc}} ( \Gamma, M )$.
Since the inflation $H^1 ( \Gamma / \Delta, M ) \rightarrow H^1 ( \Gamma, M )$ is injective, also its restriction to $H^1_{{\rm loc}} ( \Gamma / \Delta, M )$ is injective.
Thus the inflation induces an isomorphism between $H^1_{{\rm loc}}$.
\CVD

As a consequence of Lemma \ref{lem11} and Lemma \ref{lemnuovo}, we get:

\begin{pro}\label{cor12}
The local cohomology  $H^1_{{\rm loc}} ( {\mathpzc{D}}_n, \E[p^n] ) $ is trivial.
\end{pro}

\DIM Recall that $\{ Q_1, Q_2 \}$ is the chosen basis of $\E[p^n]$ in which $ \rho_n $ is diagonal.
Then the cyclic groups $\langle Q_1 \rangle$, $\langle Q_2 \rangle$ are ${\mathpzc{D}}_n$-modules.
Since $\E[p^n] \cong \langle Q_1 \rangle \bigoplus \langle Q_2 \rangle$, by Lemma~\ref{lem11}, we have
\[
H^1_{{\rm loc}} ( {\mathpzc{D}}_n, \E[p^n] ) \cong H^1_{{\rm loc}} ( {\mathpzc{D}}_n, \langle Q_1 \rangle ) \bigoplus H^1_{{\rm loc}} ( {\mathpzc{D}}_n, \langle Q_2 \rangle ).
\]
Then  it is sufficient to show that
\[
H^1_{{\rm loc}} ( {\mathpzc{D}}_n, \langle Q_1 \rangle ) = 0\quad {\rm and} \quad \ H^1_{{\rm loc}} ( {\mathpzc{D}}_n, \langle Q_2 \rangle ) = 0.
\]
We prove  $H^1_{{\rm loc}} ( {\mathpzc{D}}_n, \langle Q_1 \rangle ) = 0$. The other case is  similar.
Let
\[
C_1 =
\bigg\{
\left(
\begin{array}{cc}
1 & 0 \\
0 & \mu \\
\end{array}
\right) \in {\mathpzc{D}}_n
\bigg\}
\] be the  group of the elements of ${\mathpzc{D}}_n$  fixing $Q_1$.
By Lemma~\ref{lemnuovo},
\begin{equation}\label{rel11}
H^1_{{\rm loc}} ( {\mathpzc{D}}_n / C_1, \langle Q_1 \rangle ) \cong H^1_{{\rm loc}} ( {\mathpzc{D}}_n, \langle Q_1 \rangle ).
\end{equation}
Observe that  $C_1$ is  the kernel of the homomorphism ${\mathpzc{D}}_n \rightarrow ( \Z / p^n \Z )^\ast $ that sends
\[
\gamma =
\left(
\begin{array}{cc}
a_\gamma & 0 \\
0 & d_\gamma \\
\end{array}
\right) \in  {\mathpzc{D}}_n \mapsto a_\gamma.
\]

Then ${\mathpzc{D}}_n / C_1$ is isomorphic to a subgroup of $( \Z / p^n \Z )^\ast$ and therefore it is cyclic.
Then $H^1_{{\rm loc}} ( {\mathpzc{D}}_n / C_1, \langle Q_1 \rangle ) = 0$.
By relation~(\ref{rel11}), we get $H^1_{{\rm loc}} ( {\mathpzc{D}}_n, \langle Q_1 \rangle ) = 0$.
\CVD

\subsection{Structure of the Local Cohomology.}
\label{sec5}
We give a precise description of the elements of $H^1_{{\rm loc}} ( G_n, \E[p^n] )$.
Such a description is crucial for  the proof of our main Propositions.
\begin{pro}\label{pro13}
Suppose that $H^1 ( G_n, \E[p^n] ) \neq 0$, $\rho$ has order at least $3$ and the first eigenvalue of $\rho$ is equals to $1$.
Let $c \in H^1_{{\rm loc}} ( G_n, \E[p^n] )$.
Then there exists $\beta \in \Z/ p^n\Z$ and a cocycle $Z$ such that $[Z] = c$,
\[
Z_\tau = ( 0, 0 ), \ \forall \tau \in \langle {\mathpzc{D}}_n, {\mathpzc{sU}}_n \rangle
\]
and
\[
Z_{\tau_L} = ( 0, p^j \beta ).
\]
\end{pro}

\DIM Consider the restrictions
\[
H^1 ( G_n, \E[p^n] ) \stackrel{r_{SL}}{\rightarrow} H^1 ( \langle \rho_n, {\mathpzc{sL}}_n \rangle, \E[p^n] ),
\]
\[
H^1 ( G_n, \E[p^n] ) \stackrel{r_{SU}}{\rightarrow} H^1 ( \langle \rho_n , {\mathpzc{sU}}_n \rangle , \E[p^n] ),
\]
\[
H^1 ( G_n, \E[p^n] ) \stackrel{r_D}{\rightarrow} H^1 ( {\mathpzc{D}}_n, \E[p^n] ).
\]
Let $Z$ be a cocycle from $G_n$ to $\E[p^n]$, such that its class $[Z] \in H^1_{{\rm loc}} ( G_n, \E[p^n] )$.
If a cocycle satisfies the local conditions relative to $G_n$, then it satisfies them relative to any subgroup of $G_n$.
Thus $r_{SL} ( [Z] ) \in H^1_{{\rm loc}} ( \langle \rho_n , {\mathpzc{sL}}_n \rangle , \E[p^n] )$, $r_{SU} ( [Z] ) \in H^1_{{\rm loc}} ( \langle \rho_n , {\mathpzc{sU}}_n \rangle , \E[p^n] )$ and $r_D ( [Z] ) \in H^1_{{\rm loc}} ( {\mathpzc{D}}_n, \E[p^n] )$.
By Part 2. of Lemma \ref{lem:lem11} and by Proposition~\ref{cor12}, all such local cohomologies are trivial.
Therefore $[Z] \in \ker ( r_{SL} ) \cap \ker ( r_{SU} ) \cap \ker ( r_D )$.
In other words the restrictions of $Z$ to each subgroup $\langle \rho_n , {\mathpzc{SL}}_n \rangle$,  $\langle \rho_n , {\mathpzc{SU}}_n \rangle$ and  ${\mathpzc{D}}_n$ are coboundaries.
Hence, there exist $P, Q, R \in \E[p^n]$, such that
\[
Z_\gamma = \gamma ( P ) - P,  \hspace{0.4cm} \forall \hspace{0.2cm} \gamma \hspace{0.1cm} \in  \hspace{0.1cm} \langle \rho_n, {\mathpzc{SL}}_n \rangle;
\]
\[
Z_\delta = \delta ( Q ) - Q,  \hspace{0.4cm} \forall \hspace{0.2cm} \delta \hspace{0.1cm} \in \hspace{0.1cm}  {\mathpzc{D}}_n;
\]
\[
Z_\theta = \theta ( R ) - R,  \hspace{0.4cm} \forall \hspace{0.2cm} \theta \hspace{0.1cm} \in \hspace{0.1cm} \langle \rho_n, {\mathpzc{SU}}_n \rangle.
\]
Adding to $Z$ the coboundary $Z^\prime_\tau = \tau ( - Q ) - ( - Q )$, we get a new cocycle in the same class of $Z$ in $H^1 ( G_n, \E[p^n] )$, which is trivial over ${\mathpzc{D}}_n$.
Then, without loss of generality, we can suppose $Q = ( 0, 0 )$ and so $Z_\delta = ( 0, 0 )$ for every $\delta \in {\mathpzc{D}}_n$.
Since $\rho_n \in {\mathpzc{sL}}_n \cap {\mathpzc{sU}}_n \cap {\mathpzc{D}}_n$, we have
\[
Z_{\rho_n} = \rho_n ( P ) - P = \rho_n ( R ) - R = \rho_n ( Q ) - Q= ( 0, 0 ).
\]
Thus $P, R \in \ker ( \rho_n - I )$.
Since
\[
\rho_n =
\left(
\begin{array}{cc}
1 & 0 \\
0 & \lambda_{2, n}\\
\end{array}
\right),
\]
\noindent then $P = ( \beta, 0 )$ and $R = ( \beta^\prime, 0 )$, for some $\beta, \beta^\prime \in \Z / p^n \Z$.
By an easy computation
\[
Z_{\tau_U} = \tau_U ( \beta^\prime, 0 ) - ( \beta^\prime , 0 ) = ( 0, 0 )
\]
and
\[
Z_{\tau_L} = \tau_L ( \beta , 0 ) - ( \beta, 0 ) = ( 0, p^j \beta ).
\]
Then $Z_\tau = ( 0, 0 )$ for every $\tau \in \langle {\mathpzc{D}}_n, {\mathpzc{sU}}_n \rangle$ and $Z_{\tau_L} = ( 0, p^j \beta )$, proving the claim.
\CVD

\section{The main Propositions}\label{sec3}

In this section we collect all our results and we prove the central propositions. The proof of   Theorem \ref{sette}  realises  on the  following:
\begin{pro}\label{pro14}
Suppose that $H^1 ( G_n, \E[p^n] ) \neq 0$, the order of $\rho$ is $\geq 3$, $\rho$ has the first eigenvalue equals to $1$ and $G_1$ is not cyclic.
Then the cohomological group $H^1_{{\rm loc}} ( G_n, \E[p^n] ) = 0$.
\end{pro}

\DIM By~\cite[Lemma 14]{P-R-V2}, $\tau_U = {1\,\,\, 1 \choose 0\,\,\, 1}$.
By a simple calculation
\[
\delta = \tau_U \tau_L \tau_U^{-( p^j + 1 )^{-1}} \tau_L^{-( p^j + 1 )}=
\left(
\begin{array}{cc}
1 + p^j  & 0 \\
0 & 1 - p^j ( p^j + 1 )^{-1} \\
\end{array}
\right)
\in G_n.
\]
Let $Z$ be a cocycle such that $[Z] \in H^1_{{\rm loc}} ( G_n, \E[p^n] )$.
By Proposition~\ref{pro13}, we can suppose that $Z_\gamma = ( 0, 0 )$ for every $\gamma \in \langle {\mathpzc{sU}}_n, {\mathpzc{D}}_n \rangle$ and $Z_{\tau_L} = ( 0, p^j \beta )$, for a certain $\beta \in \Z / p^n \Z$.
Then, in particular $Z_{\delta \tau_U^b} = ( 0, 0 )$ for every $b \in \Z / p^n \Z$.
By the cocycle conditions,
\[
Z_{\tau_L \delta \tau_U^b} = Z_{\tau_L} + \tau_L Z_{\delta \tau_U^b} = Z_{\tau_L} = ( 0, p^j \beta ).
\]
We now use the fact that $Z$ satisfies the local conditions, to prove that $p^j \beta \equiv 0 \mod ( p^n )$.
By a short computation
\[
\tau_L \delta  \tau_U^b =
\left(
\begin{array}{cc}
1 + p^j  & ( 1 + p^j ) b \\
p^{2j} + p^j & 1 + ( p^{2j} + p^j ) b - p^j ( p^j + 1 )^{-1} \\
\end{array}
\right)
.
\]
Since $Z$ satisfies the local conditions, for every $b \in \Z / p^n \Z$, there exist $x, y \in \Z / p^n \Z$ such that $( \tau_L \delta \tau_U^b - I ) ( x, y ) = Z_{\tau_L \delta \tau_U^b} = ( 0, p^j \beta )$.
Thus
\[
\left(
\begin{array}{cc}
p^j  & ( 1 + p^j ) b \\
p^{2j} + p^j & ( p^{2j} + p^j ) b - p^j ( p^j + 1 )^{-1} \\
\end{array}
\right)
\left(
\begin{array}{cc}
x  & \\
y & \\
\end{array}
\right)
=
\left(
\begin{array}{cc}
0  &  \\
p^j \beta & \\
\end{array}
\right).
\]
In particular, for $b \equiv -p^j ( p^j + 1 )^{-2} \mod ( p^n )$, we have
\begin{equation*}
\begin{cases}
p^j x -p^j ( p^j + 1 )^{-1} y \equiv 0 \mod ( p^n )\\
( p^j + 1 ) ( p^j x -p^j ( p^j + 1 )^{-1} y ) \equiv p^j \beta \mod ( p^n ).\end{cases}
\end{equation*}

\noindent Whence  $p^j \beta \equiv 0 \mod ( p^n )$.
Thus $Z_{\tau_L} = ( 0, 0 )$, and so $Z_\tau = ( 0, 0 )$ for every $\tau \in \langle {\mathpzc{sU}}_n, {\mathpzc{sL}}_n, {\mathpzc{D}}_n \rangle$, which is $G_n$, by Lemma \ref{lem:lem11}.
Then $Z$ is a coboundary.
\CVD

The proof of   Theorem \ref{cinque}  realises  on the  following:

\begin{pro}\label{pro15}
Suppose that $H^1_{{\rm loc}} ( G_n, \E[p^n] ) \neq 0$, $\rho$ has order at least $3$ and the first eigenvalue of $\rho$ is equals to $1$.
Then $G_2$ is contained either in the group of the lower triangular matrices, or in the group of the upper triangular matrices.
\end{pro}

\DIM By Proposition~\ref{pro14} we can suppose that $G_1$ is cyclic, generated by $\rho$.
Recall that  $G_2 = \langle \rho_2 ,  \Gal ( K_2 / K_1 ) \rangle$, $\rho_2$ is diagonal and all the elements of $\Gal ( K_2/ K_1 )$ are congruent  to the identity $I$ modulo $p$. Suppose that $G_2$ is contained neither in the group of the upper triangular matrices nor in the group of the lower triangular matrices.
 Then there exist $a, b, c, d \in \Z / p \Z$ with $b, c \in \Z/ p\Z^\ast$ such that ${1 + p a\,\,\, p b \choose p c\,\,\, 1 + p d} \in G_2$.
Since $G_2$ is generated by strictly upper and lower triangular and diagonal matrices,  $ {1 \,\,\, p \choose 0\,\,\, 1 }\in G_2$ and $ {1 \,\,\, 0 \choose p \,\,\, 1 } \in G_2$.
We now prove that $\tau_U = {1 \,\,\, p \choose 0\,\,\, 1 }$. Similarly one proves   $\tau_L = {1 \,\,\, 0 \choose p \,\,\, 1 }$.
Let $\tau_n$ be a lift of ${1 \,\,\, p \choose 0\,\,\, 1 }$ to $G_n$.
By Lemma~\ref{lem:lem11} 1., $\tau_n$ decomposes as a product of diagonal, strictly upper triangular and strictly lower triangular matrices.
Since $\tau_n$ reduces to a strictly upper triangular matrix modulo $p^2$, at least one of its factors is of the type
\[
\gamma =
\left(
\begin{array}{cc}
1 & p e \\
0 & 1 \\
\end{array}
\right) \,\,\,\,{\rm with}\,\,\,\,  e \not\equiv 0 \mod ( p ).
\]
Then $\gamma^{e^{-1}} = {1 \,\,\, p \choose 0\,\,\, 1 } \in G_n$.

Let $h \geq 1$ be the  minimal natural number such that there exists a matrix
\[
\delta =
\left(
\begin{array}{cc}
1 + p^h a & 0 \\
0 & \mu + p d \\
\end{array}
\right)  \in {\mathpzc{D}}_n
\]
with $a \in \Z / p^{n-h}\Z^\ast$, $d \in \Z / p^{n-1} \Z$ and $\mu$ congruent to a power of $\lambda_2$ modulo $p$.
By eventually  replacing $\delta$ with a suitable power of it times a suitable power of $\rho_n$, we can suppose $a = 1$ and $\mu = 1$.
Let $Z$ be a cocycle such that $[Z] \in H^1_{{\rm loc}} ( G_n, \E[p^n] )$.
We now compute $Z_{\tau_L^{p^{h-1}} \delta \tau_U^b}$, for every $b \in \Z / p^{n-1} \Z$.
By Proposition~\ref{pro13}, we can suppose
\[
Z_\gamma = ( 0, 0 ), \ \forall \gamma \in \langle {\mathpzc{D}}_n, {\mathpzc{sU}}_n \rangle
\]
and
\[
Z_{\tau_L} = ( 0, p \beta ),
\]
for a certain $\beta \in \Z / p^n \Z$.
Then
\begin{align*}
Z_{\tau_L^{p^{h-1}} \delta \tau_U^b} & = Z_{\tau_L^{p^{h-1}}} + \tau_L^{p^{h-1}} ( Z_{\delta \tau_U^b} ) \\
& = \tau_L^{p^{h-1}} ( ( \beta, 0 ) ) - ( \beta, 0 ) \\
& = ( 0, p^h \beta ).
\end{align*}
We shall prove that   $[Z]$ is a coboundary. To this purpose we show that  the local conditions imply  $Z_\tau = \tau ( ( \beta, 0 ) ) - ( \beta, 0 )$, for every $\tau \in G_n$.
By the local conditions, for every $b \in \Z / p^{n-1} \Z$, there exist $x, y \in \Z / p^n \Z$, such that $( \tau_L^{p^{h-1}} \delta \tau_U^b - I ) ( ( x, y ) ) = Z_{\tau_L^{p^{h-1}} \delta \tau_U^b} = ( 0, p^h \beta )$.
In other words
\[
\left(
\begin{array}{cc}
p^h  & ( p + p^{h+1} ) b \\
p^h + p^{2h} & p d + ( p^{h+1} + p^{2h+1} ) b \\
\end{array}
\right)
\left(
\begin{array}{cc}
x  & \\
y & \\
\end{array}
\right)
=
\left(
\begin{array}{cc}
0  &  \\
p^h \beta & \\
\end{array}
\right).
\]
For $b \equiv d ( p^h + 1 )^{-1} \mod ( p^{n-1} )$, this gives:
\begin{equation*}
\begin{cases}
p^h x + p d y \equiv 0 \mod ( p^n )
\\
( p^h + 1 ) ( p^h x + p d y ) \equiv p^h \beta \mod ( p^n ).\end{cases}
\end{equation*}

\noindent Whence $p^h \beta \equiv 0 \mod ( p^n )$.
By minimality  of $h$, for every diagonal matrix $\delta^\prime \in {\mathpzc{D}}_n$, we have $( \delta^\prime - I ) ( ( \beta, 0 ) ) = 0$.
Thus $Z_\tau = \tau ( (\beta, 0 ) ) - ( \beta, 0 )$ over the group $\langle \tau_U, {\mathpzc{D}}_n, \tau_L \rangle$, which is $G_n$, by Lemma~\ref{lem:lem11} 1.
\CVD

\section{An example}
\label{sec6}
In this section we construct an example which shows that the assumption that $k$ does not contain $\Q ( \zeta_p + \overline{\zeta_p} )$ in our Main Theorem is necessary. The example is inspired by \cite{DZ} example 3.4 page 327. However their example only shows that the weaker condition  that $k$ does not contain $\Q ( \zeta_p )$ is necessary.
Let $p$ be an odd prime number.
Consider the following subgroup of $\GL_2 ( \Z / p^2 \Z )$:
\[
G_2 = \bigg\langle
\delta_1 =
\left(
\begin{array}{cc}
1 & 0 \\
0 & -1 \\
\end{array}
\right),
\delta_2 =
\left(
\begin{array}{cc}
1 + p & 0 \\
0 & 1 + p \\
\end{array}
\right),
\delta_3 =
\left(
\begin{array}{cc}
1 & mp \\
p & 1 \\
\end{array}
\right)
\bigg\rangle
\]
where $m$ is an integer that is not a square modulo $p$.
Observe that $\langle \delta_2, \delta_3 \rangle$ is abelian and it is normal in $G_2$.
Moreover the order of $\delta_1$ is $2$, while $\delta_2$ and $\delta_3$ have order $p$.
Then, for every $g \in G_2$, there exist a unique triple $( a, b, c ) \in \Z/ 2 \Z \times \Z/ p \Z \times \Z / p \Z$, such that $g(a,b,c) = \delta_1^a \delta_2^b \delta_3^c$. By a simple computation,
\[
g ( a, b, c ) =
\left(
\begin{array}{cc}
1 + pb & mpc \\
( -1 )^a p c & ( - 1 )^a ( 1 + pb ) \\
\end{array}
\right).
\]
Moreover, for every two triplets $( a_1, b_1, c_1 ), ( a_2, b_2, c_2 ) \in \Z/ 2 \Z \times \Z/ p \Z \times \Z / p \Z$, we have
\begin{equation}\label{relsomma}
g ( a_1, b_1, c_1 ) g ( a_2, b_2, c_2 ) = g ( a_1 + a_2, b_1 + b_2, ( -1 )^{a_2} c_1 + c_2 ).
\end{equation}

\subsection{The local cohomology.}
\label{sec61}

We now define a map $Z \colon G_2 \rightarrow ( \Z/ p^2 \Z )^2$, by sending $g ( a, b, c )$ to $Z_{g ( a, b, c )} = ( 0, ( -1 )^a p c )$.
We show that $Z$ is a cocycle and its class is a non-trivial element of $H^1_{{\rm loc}} ( G_2, ( \Z / p^2 \Z )^2 )$.
By relation (\ref{relsomma}) we have $Z_{g ( a_1, b_1, c_1 ) g ( a_2, b_2, c_2 )} = ( 0, ( -1 )^{a_1 + a_2} p ( (-1 )^{a_2} c_1 + c_2 ) )$.
On the other hand a short computation shows that
\[
Z_{g ( a_1, b_1, c_1 )} + g ( a_1, b_1, c_1 ) Z_{g ( a_2, b_2, c_2 )} = ( 0, ( -1 )^{a_1 + a_2} p ( (-1 )^{a_2} c_1 + c_2 ) ).
\]
Thus $Z$ is a cocycle.

We now prove that the class of $Z$ is in the local cohomological group. This is equivalent to show that for every $g ( a, b, c ) \in G_2$, there exist $( x, y ) \in \Z / p^2 \Z$ such that
\[
Z_{g ( a, b, c )} = ( g ( a, b, c ) - I ) ( x, y ).
\]
This relation is equivalent to the system
\[
\left(
\begin{array}{cc}
pb  & pcm \\
( -1 )^a p c & ( - 1 )^a ( 1 + b p ) -1 \\
\end{array}
\right)
\left(
\begin{array}{cc}
x  & \\
y & \\
\end{array}
\right)
=
\left(
\begin{array}{cc}
0  &  \\
( -1 )^a pc & \\
\end{array}
\right),
\]
which gives
\begin{equation*}
\begin{cases}
p b x + p c m y\equiv 0 \mod ( p^2 ) \\
( - 1 )^a p c x + ( ( -1 )^a ( 1 + b p ) - 1 ) y \equiv ( - 1 )^a p c \mod ( p^2 ).
\end{cases}
\end{equation*}
Suppose first $a = 1$.
Then the vector $( 0, pc ( pb + 2 )^{-1} )$ is a solution of the system.
Suppose now $a = 0$.
If $c \equiv 0 \mod ( p )$ then  $( 0, 0 )$ is the solution.
If $c \not \equiv 0 \mod ( p )$, a solution of the system is $( c^2m ( c^2m - b^2 )^{-1}, -bc ( c^2m - b^2 )^{-1} )$.
Thus the class of $Z$ is in $H^1_{{\rm loc}} ( G_2, ( \Z / p^2 \Z )^2 )$.
Finally let us show that $Z$ is not a coboundary.
This is equivalent to the fact that the solutions of the system are dependent on the triplets $( a, b, c )$.
For the triplets $( 0, 1, 0 )$, any solution is congruent to the vector $( 0, 0 )$ modulo $p$.
For the triplets $( 0, 0, 1 )$, any solution is congruent to the vector $( 1, 0 )$ modulo $p$.
Then $Z$ is not a coboundary.

\subsection{The example}\label{sec7}

Let $L$ be a number field and let $\E$ be a non $CM$ elliptic curve defined over $L$.
Then by \cite{Ser} there exists a constant $c ( L, \E )$ depending on $L$ and $\E$ such that, for every prime number $p > c ( L, \E )$ and for every positive integer $n$, $\Gal ( L ( \E[p^n] ) / L ) \cong \GL_2 ( \Z / p^n \Z )$.
By Galois correspondence, for every prime $p > c ( L, \E )$ and every subgroup $H_n$ of $\GL_2 ( \Z / p^n \Z )$,  there exists a field $L^\prime$ such that $L \subseteq L^\prime \subseteq L ( \E[p^n] )$ and $\Gal ( L ( \E[p^n] ) / L^\prime ) \cong H_n$.
Since $L ( \E[p^n] ) = L^\prime ( \E[p^n] )$, we get $\Gal ( L^\prime ( \E[p^n] ) / L^\prime ) \cong H_n$.

Let $p > c ( L, \E )$ be a prime and identify $\Gal ( L ( \E[p^n] ) / L )$ with $\GL_2 ( \Z / p^n \Z )$.
Then, in particular, we can find a field $L^\prime$ such that $\Gal ( L^\prime ( \E[p^2] ) / L^\prime ) = G_2$.
Thus the $G_2$-module $\E[p^2]$ is isomorphic to $G_2$-module $( \Z / p^2 \Z )^2$.
By the previous subsection, $H^1_{{\rm loc}} ( {\rm Gal} ( L^\prime ( \E[p^2] / L^\prime ), \E[p^2] )$ is not trivial.
Moreover, by the Cohomological Criterion (see the introduction), there exists a field $k$ containing $L^\prime$ such that $k \cap L^\prime ( \E[p^2] ) = L^\prime$ and there is a counterexample to the local-global divisibility by $p^2$ over $k$.

We now show  that $k$ contains $\Q ( \zeta_p + \overline{\zeta_p} )$, but it does not contain $\Q ( \zeta_p )$.
First observe that, since $k \cap L^\prime ( \E[p^n] ) = L^\prime$, we have $\Gal ( k ( \E[p^n] ) / k ) \cong G_2$.
 By definition of $G_2$ (see the previous subsection), $\Gal ( k ( \E[p] )/ k )$ is a cyclic group of order $2$ generated by the matrix ${1\,\,\, 0 \choose 0 \,\,\, -1}$.
Since $k ( \E[p] )$ contains $\Q ( \zeta_p )$ and the elements of $\Gal ( k ( \E[p] )/ k )$ that fix $\zeta_p$ are exactly the elements with determinant $1$, we have $k ( \E[p] ) = k ( \zeta_p )$.
Thus $[k ( \zeta_p ) : k] = 2$, and so $k$ does not contain $\Q ( \zeta_p )$ and $k$ contains $\Q ( \zeta_p + \overline{\zeta_p} )$.

We finally show that every elliptic curve $k$-isogenous to $\E$ does not admit a cyclic $k$-isogeny of degree $p^3$.
Observe that $\E[p^2]$ has two cyclic distinct $G_2$-modules contained in $\E[p]$ that correspond to the eigenvectors of the eigenvalues of ${1\,\,\, 0 \choose 0 \,\,\, -1}$, but it has no cyclic $G_2$-module of order $p^2$. Suppose that $\E_1, \E_2$ are two elliptic curves in the $k$-isogeny class of $\E$ admitting a cyclic $k$-isogeny $\phi$ of degree $p^3$. Let $\phi_1 \colon \E \rightarrow \E_1$ be a $k$-isogeny. Then the kernel of the composition  $\phi \circ\phi_1$ has a  $G_2$-submodule of degree $p^2$ or $p^3$, which is a contradiction.

\newpage

\thispagestyle{empty}

Laura Paladino\par\smallskip
Dipartimento di Matematica \par
Universit\`{a} della Calabria\par\smallskip
via Ponte Pietro Bucci, cubo 31b\par
IT-87036 Arcavacata di Rende (CS)\par
e-mail address: paladino@mat.unical.it

\vskip 1cm

Gabriele Ranieri\par\smallskip
Scuola Normale Superiore\par
Piazza dei Cavalieri 7,\par\smallskip
56100 Pisa,\par
Italy\par
e-mail address: ranieri@math.unicaen.fr

\vskip 1cm

Evelina Viada\par\smallskip
Departement Mathematik\par
Universit\"{a}t Basel\par\smallskip
Rheinsprung, 21\par
CH-4051 Basel\par
e-mail address: evelina.viada@unibas.ch


\begin{thebibliography}{Pal}


\bibitem[DZ]{DZ}  \textsc{Dvornicich R., Zannier U.}, \emph{Local-global divisibility of rational points in some commutative algebraic   groups}, Bull. Soc. Math. France, \textbf{129} (2001), 317-338.

\bibitem[DZ2]{DZ2}  \textsc{Dvornicich R., Zannier U.}, \textit{An analogue for elliptic curves of the Grunwald-Wang example}, C. R. Acad. Sci. Paris, Ser. I \textbf{338} (2004), 47-50.

\bibitem[DZ3]{DZ3}  \textsc{Dvornicich R.}, \textsc{Zannier U.}, \textit{On local-global principle for the divisibility of a rational point by a positive integer}, Bull. Lon. Math. Soc., no. \textbf{39} (2007), 27-34.


\bibitem[Ken]{Ken}  \textsc{Kenku M. A.}, \textit{On the modular curves $X_0 ( 125 ), X_1 ( 25 )$ and $X_1 ( 49 )$}, J. London Math. Soc. (2)  {\bf 23}, no. 3, (1981), 415–427.






\bibitem[Maz]{Maz} \textsc{Mazur B.}, {\it Rational isogenies of prime degree (with an appendix of D. Goldfeld)}, Invent. Math., no. {\bf 44}, (1978), 129-162.

\bibitem[Maz2]{Maz2} \textsc{Mazur B.}, {\it Modular curves and the Eisenstein Ideal}, Publ. Math.,  Inst. Hauter Etud. Sci , no. {\bf 47}, (1978), 33-186.

\bibitem[Mer]{Mer} \textsc{Merel L.}, \textit{Bornes pour la torsion des courbes elliptiques sur les corps de nombres}, Invent. Math. {\bf 124} no. 1-3 (1996), 437-449.

\bibitem[Mer2]{Mer2}  Merel L., \textit{Sur la nature non-cyclotomique des points d'ordre fini des
      courbes elliptiques}. Duke Math. J.  \textbf{110} no. 1, (2001), 81-119.





\bibitem[Pal]{Pal2}  \textsc{Paladino L.},  \textit{On counterexamples to local-global divisibility in commutative algebraic groups}, Acta Arithmetica, 
\textbf{148} no. 1 (2011), pp. 21-29.


\bibitem[PRV]{P-R-V}  \textsc{Paladino L., Ranieri G., Viada E.},  \textit{Local-Global Divisibility by $p^2$ in elliptic curves}, preprint.

\bibitem[PRV2]{P-R-V2}  \textsc{Paladino L., Ranieri G., Viada E.},  \textit{Local-Global Divisibility by $p^n$ in elliptic curves}, preprint.


\bibitem[Ser]{Ser} \textsc{Serre J-P.}, {\it Propriet\'es Galoisiennes des points d'ordre fini des courbes elliptiques}, Invent. Math. no. {\bf 15}, (1972) 259-331.



\bibitem[Sil]{Sil} \textsc{Silverman J. H.}, \textit{The arithmetic of elliptic curves, $2$nd edition}, Springer, Heidelberg, 2009.


\end{thebibliography}
\end{document}